\documentclass[
a4paper
,11pt
]{amsart}

\usepackage{tikz}
\usetikzlibrary{decorations.pathreplacing}

\usepackage[section]{placeins}
\usepackage{amssymb,latexsym}
\usepackage[utf8]{inputenc}
\usepackage{amsmath}
\usepackage{amsfonts}
\usepackage{dsfont}
\usepackage{amsthm}
\usepackage{thmtools}
\usepackage{mathrsfs} 

\usepackage{enumerate}
\usepackage{ifthen}
\usepackage{comment}
\usepackage{xcolor}
\usepackage{mathtools}
\usepackage{booktabs}
\usepackage{etoolbox}
\usepackage{libertine}
\usepackage[T1]{fontenc}
\usepackage[libertine]{newtxmath}
\usepackage{newtxmath}

\declaretheorem[parent=section]{theorem}
\declaretheorem[sibling=theorem]{proposition}
\declaretheorem[sibling=theorem]{lemma}
\declaretheorem[sibling=theorem]{corollary}

\declaretheorem[sibling=theorem,style=remark]{remark}
\declaretheorem[sibling=theorem,style=remark]{example}
\declaretheorem[sibling=theorem,style=definition]{definition}

\declaretheorem[numbered=no,name=Theorem]{theorem*}
\declaretheorem[numbered=no,style=definition,name=Definition]{definition*}
\declaretheorem[numbered=no,name=Proposition]{proposition*}

\usepackage{hyperref}
\hypersetup{
   backref
   ,colorlinks=true
   ,allcolors=black
}

\begin{document}
\bibliographystyle{amsalpha}

\newcommand{\mi}{\ensuremath{\mathrm{i}}}
\newcommand{\me}{\ensuremath{\mathrm{e}}}
\newcommand{\Id}{\on{Id}}
\newcommand{\on}[1]{\operatorname{#1}}
\newcommand{\abs}[1]{\left\vert #1 \right\rvert}
\newcommand{\norm}[1]{\left\lVert #1 \right\rVert}
\newcommand{\Ex}[1]{\operatorname{E}\! \left[ #1 \right]}
\newcommand{\diff}[1]{\operatorname{d}\ifthenelse{\equal{#1}{}}{\,}{#1}}
\newcommand{\Ker}[1]{\operatorname{ker}\left(#1\right)}
\newcommand{\E}{\operatorname{E}}
\newcommand{\C}{\mathbb{C}}
\newcommand{\R}{\mathbb{R}}
\newcommand{\Q}{\mathbb{Q}}
\newcommand{\N}{\mathbb{N}}
\newcommand{\Z}{\mathbb{Z}}
\newcommand{\LL}{\on{L}}

\newcommand{\todo}[1]{\textcolor{red}{\textbf{TODO:}$\left\{\right.$#1$\left.\right\}$}}

\author{S. Bourguin}
\address{Boston University, Department of Mathematics and Statistics, 111
  Cummington Mall, Boston, MA 02215, USA}
\email{bourguin@math.bu.edu}
\author{S. Campese}
\address{University of Luxembourg, Mathematics Research Unit, 6, rue Richard
  Coudenhove-Kalergi, 1359 Luxembourg, Luxembourg}
\email{simon.campese@uni.lu}
\author{N. Leonenko}
\address{Cardiff University, School of Mathematics, Senghennydd Road, Cardiff, Wales, UK, CF24 4AG}
\email{leonenkon@cardiff.ac.uk}
\author{M.S. Taqqu}
\address{Boston University, Department of Mathematics and Statistics, 111
  Cummington Mall, Boston, MA 02215, USA}
\email{murad@bu.edu}
\title{Four moments theorems on Markov chaos}
\thanks{S. Campese was supported by InterMobility grant LILAC of the Luxembourg
  National Research Fund. M.S. Taqqu and N. Leonenko were supported by Cardiff Incoming
  Visiting Fellowship Scheme. N. Leonenko was supported partly by Australian Research Council's Discovery Projects funding scheme (project DP160101366) and by project MTM2015-71839-P of MINECO, Spain (co-funded with FEDER funds)}
\begin{abstract}

  We obtain quantitative Four Moments Theorems establishing convergence of the laws
  of elements of a Markov chaos to a Pearson distribution, where the only assumption we make on the Pearson distribution is
  that it admits four moments. While in general one cannot use moments
  to establish convergence to a heavy-tailed distributions, we provide a context
  in which only the first four moments suffices. These results are
  obtained by proving a general carr\'e du champ bound on the distance
  between laws of random variables in the
  domain of a Markov diffusion generator and invariant measures of diffusions. For elements of a Markov
  chaos, this bound can  be reduced to just the first four moments.
\end{abstract}
\maketitle

\section{Introduction}

Four Moments Theorems are results which imply or characterize convergence
in law of some approximating sequence $\left\{ F_k\colon k \geq 0 \right\}$ of random variables
towards some target measure $\nu$. A typical example of such an
approximating sequence (with the target measure $\nu$ being Gaussian) are
homogeneous sums of the form
\begin{equation}
  \label{eq:37}
F_k = \sum_{j_1, \ldots , j_p=1}^k a_{j_1\cdots j_p}W_{j_1}\cdots W_{j_p},
\end{equation}
normalized to have unit variance. Here, $\left\{ W_j\colon j \geq 1 \right\}$ is an i.i.d. sequence of
standard Gaussian random variables and the constants $a_{j_1\cdots j_p}$ are
symmetric in the indices and vanish on diagonals. The classical fourth moment theorem of Nualart and Peccati (see
\cite{nualart_central_2005}) states
that $F_k$ converges in law to a standard Gaussian distribution if and
only if the fourth moment of $F_k$ converges to the fourth moment of
the standard Gaussian distribution, namely 3. In fact, the
aforementioned two authors have proven their result in infinite
dimensions, where the sequence of $F_k$ are sequences of multiple Wiener-It\=o
integrals of fixed order $p$. The original proof in~\cite{nualart_central_2005}
uses stochastic analysis and shortly after its publication another proof via
Malliavin calculus was given by Nualart and Ortiz-Latorre
in~\cite{nualart_central_2008}. Later, in \cite{nourdin_noncentral_2009},
Nourdin and Peccati used this approach to obtain a similar result for
approximation  of the Gamma distribution. They showed, again for a sequence of
normalized Wiener-It\=o integrals, that convergence of the third and fourth  moments is enough to converge to a Gamma distribution.

In \cite{ledoux_chaos_2012}, Ledoux gave new proofs of the above results from
the abstract point of view of Markov diffusion generators. In this context, given a Markov diffusion generator satisfying a certain spectral
condition, convergence of a sequence of eigenfunctions of such a
generator to a Gaussian or Gamma distribution is still controlled by
convergence of just the first four moments. Multiple Wiener-It\=o integrals fit
the framework as they are eigenfunctions of the infinite-dimensional
Ornstein-Uhlenbeck operator.

Building on \cite{ledoux_chaos_2012}, Azmoodeh et al. showed in \cite{azmoodeh_fourth_2014} that
the spectral condition can be replaced with a Markov chaos
property of the eigenfunctions which is less
restrictive than the earlier notion of Markov chaos introduced in \cite{ledoux_chaos_2012}. In addition to Four Moments Theorems for convergence
towards the Gaussian and Gamma distributions, a Four Moments Theorem
for the approximation of the Beta distribution was proven. 

When viewed from the point of view of diffusion theory, it is interesting to
note that the Gaussian, Gamma and Beta distributions share a common
feature: they are the only invariant measures of a diffusion on the real line admitting
an orthonormal basis of polynomial eigenfunctions (see \cite{mazet_classification_1997}). They
also are the only members of the Pearson family of distributions (introduced in
\cite{pearson_contributions_1895}, see for example
\cite{forman_pearson_2008} for a modern treatment) which
have moments of all orders. This naturally leads to the question whether Four
Moments Theorems can also be proven when the target measure $\nu$ is
one of the three remaining heavy-tailed
classes of Pearson distributions, commonly known as skew-t, F- and inverse
Gamma-distributions (see Subsection \ref{pearsondiffsubsec} for details). Here, heavy-tailed is
understood in the sense that only a finite number of moments exist.  

Distributions $\nu$ belonging to the Pearson family (or more generally absolutely
continuous invariant measures of diffusions) have already been
considered in \cite{eden_nourdin-peccati_2015,kusuoka_steins_2012} as
possible limit laws $\nu$ for sequences of multiple
Wiener-It\=o integrals. These integrals have the infinite dimensional
Ornstein-Uhlenbeck as underlying Markov generator. For such
multiple Wiener-It\=o integrals, however, as was also observed in
\cite{kusuoka_steins_2012}, the only possible limit distributions belonging to the
Pearson family are the Gaussian and Gamma laws.

In this paper, we present a systematic approach to the problem and prove
quantitative
Four Moments Theorems for all six classes of the Pearson distribution. The only
assumption we make, which seems to be unavoidable in this context, is that the parameters of the
distribution are chosen in such a way that the first four moments exist. Compared to
\cite{azmoodeh_fourth_2014}, we are not only able to cover the full Pearson class as a target
distribution $\nu$, but also extend the admissible chaos structures, so
that, for example, the laws of the converging sequence of chaotic random variables can be
heavy-tailed as well. In particular, no assumption of hypercontractivity or
diagonalizability of the underlying generator is made. Our main result
(Theorem \ref{thm:1} along with Proposition \ref{prop:1}, to which we refer 
for full details and any unexplained notation) is a quantitative Four Moments Theorem of
the form
\begin{equation}
  \label{eq:40}
  d(G_k,Z) \leq c \sqrt{ \int_E^{} P(G_k) \diff{\mu} + \xi_k \int_E^{} Q(G_k) \diff{\mu}},
\end{equation}
where $d$ is a suitable probabilistic distance metrizing convergence in law, $c$
is a positive constant, $Z$
is a random variable whose law $\nu$ belongs to the Pearson family,
$G_k$ is a chaotic random variable with some underlying  Markov diffusion
generator $L_k$ for each $k \geq 0$, and
$P$ and $Q$ are polynomials of degree four, whose coefficients are explicitly
given in terms of the parameters of the law of $Z$. In comparison to earlier
Four Moments Theorems, the linear combination of moments, given by the
integral involving the polynomial $Q$ appearing in the bound on
the right hand side of~\eqref{eq:40} is
new (and only appears in certain cases). The deterministic non-negative
real sequence $\left\{ \xi_k\colon k\geq 0 \right\}$ is defined in terms of a new notion of \emph{chaos grade},
which, heuristically speaking, measures how similar the chaotic sequence
$\left\{ G_k\colon k\geq 0 \right\}$ is to the target random variable $Z$, when the latter is viewed as
an element of the Markov chaos of a Pearson generator. 
To prove~\eqref{eq:40}, we first obtain a generic bound of the form
\begin{equation}
  \label{eq:41}
  d(G,Z) \leq c \int_E^{} \abs{\Gamma(G,-L^{-1}G) - \tau(G)} \diff{\mu}
\end{equation}
(see Theorem \ref{thm:4}) for probabilistic
distances $d(G,Z)$ between a target random variable $Z$ whose law can belong to
a large class of absolutely continuous distributions, and a random
variable $G$ coming from a Markov structure which involves the \emph{carré du
  champ} operator $\Gamma$, the pseudo inverse $L^{-1}$ of the underlying Markov
generator and a function $\tau$ related to the target measure. Again, we stress
that both, the laws of $Z$ and $G$ do not need to have moments of all orders.
The bound \eqref{eq:41} is of independent interest and obtained using a combination
of Stein's method and the so-called Gamma calculus. It can be seen as an
abstract version of the \emph{Malliavin-Stein method} on Wiener chaos, first
introduced in \cite{nourdin_steins_2009}. Then, in order to further  bound~\eqref{eq:41} by the right hand side of~\eqref{eq:40} when $G$ is
a chaotic element of the Markov structure and the law of $Z$ belongs to the Pearson
family, we again make use of the Gamma calculus and spectral
arguments that, in a similar spirit as in \cite{azmoodeh_fourth_2014},
allow us to obtain a linear combinations of the first four moments as
a bound for the right hand side of \eqref{eq:41}.

Note that, in general, one cannot a priori use moments to prove
convergence towards a heavy-tailed distribution. However, our results
provide a context in which this is not only possible, but where convergence
of only the first four moments  suffices.

As particular examples of structures fitting our framework, we study
(tensorized) Pearson generators, which have multivariate Pearson distributions
as invariant measures. In this context, the chaos grade provides a
heuristic for the question which Pearson laws are compatible with each
other, in the sense that
chaotic random variables (for example homogeneous sums of the type~\eqref{eq:37}
with the Gaussian laws replaced by arbitrary other Pearson laws with
finite first four moments) with respect to one Pearson generator
can converge in distribution to the invariant measure of another (possibly
different) Pearson generator.

The paper is organized as follows. In Section \ref{s-11}, we introduce
the Markov framework we will be working in and give a quick summary of Stein's
method as well as an overview of the Pearson distributions. Our main results, in
particular the bounds~\eqref{eq:41} and~\eqref{eq:40} as well as the definition
of Markov chaos are presented in Section \ref{mainresultssec}. As an example, we
study in Section \ref{secchaosstruct} the case of Pearson chaos, whose chaos structure fits our framework.

As a last remark, let us mention that speaking of Four Moments Theorems as
opposed to a Fourth Moment or generally a Third and Fourth Moment is
merely a
question of style, depending on whether one normalizes the approximating
sequences to have the correct mean and variance or not. We chose not to impose any
normalization.

\section{Preliminaries}
\label{s-11}

\subsection{Markov diffusion generators}
\label{s-10}

Our main results will be proven in the setting of Markov diffusion generators,
that is, we have some underlying diffusive Markov process $\left\{
  X_t\colon t\geq 0 \right\}$ with
invariant measure $\mu$, associated semigroup $\left\{ P_t\colon t\geq
0\right\}$, infinitesimal
generator $L$ and carré du champ $\Gamma$, where all of these objects
are inherently connected. The operators $L$ and $\Gamma$ play an
important role here. From an abstract
point of view, a standard and elegant way to introduce this setting is through so
called Markov triples, where one starts from the invariant measure $\mu$, the
carré du champ $\Gamma$ and a suitable algebra of functions (random
variables), from which the generator $L$, the semigroup $\left\{ P_t\colon t\geq
0\right\}$ (including
their $L^2$-domains) and thus also
the Markov process $\left\{ X_t\colon t\geq
0\right\}$ are constructed.
The assumptions we will make here are those of a so-called \emph{Full Markov Triple} $(E,\mu,\Gamma)$ in the sense of \cite[Part I,
Chapter 3]{bakry_analysis_2014}. Before introducing this setting rigorously, let us give an informal description.
Random variables are 
viewed as elements of an algebra $\mathcal{A}$ of functions $F \colon E \to \R$,
where $(E,\mathcal{F},\mu)$ is some probability space. On this algebra, the
generator $L$ and the bilinear and symmetric carré du champ operator $\Gamma$ are defined and related via the identity
\begin{equation*}
  \Gamma(F,G) = \frac{1}{2} \left( L (FG) - FLG - GLF \right).
\end{equation*}
They satisfy a diffusion property, which in its simplest form reads
\begin{equation*}
  L \varphi(F) = \varphi'(F) LF + \varphi''(F) \Gamma(F,F)
\end{equation*}
or, expressed using the carré du champ,
\begin{equation*}
  \Gamma(\varphi(F),G) = \varphi'(F) \Gamma(F,G).
\end{equation*}
The subset of random variables with finite mean and variance is then $L^2(E,\mu)
\subseteq \mathcal{A}$. On this smaller space, $L$ and $\Gamma$ are typically only densely
defined on their domains $\mathcal{D}(L)$ and $\mathcal{D}(\mathcal{E}) \times
\mathcal{D}(\mathcal{E})$. The symbol $\mathcal{E}$, defined below, stands for a Dirichlet form
(the so-called energy functional), which is used to construct the
domains. On these domains, an important relation between $L$ and
$\Gamma$ holds, namely the integration by parts formula
\begin{equation*}
  \int_E^{} \Gamma(F,G) \diff{\mu} = - \int_E^{} FLG.
\end{equation*}

We are now going to introduce this setting in a
rigorous way, following closely~\cite[Part I, Chapter
3]{bakry_analysis_2014}. The needed definitions and assumptions are as follows. 
\begin{enumerate}[(i)]
\item $(E,\mathcal{F},\mu)$ is a probability space and $L^2(E,\mathcal{F},\mu)$ is separable.
\item $\mathcal{A}$ is a vector space of real-valued, measurable functions
  (random variables) on $(E,\mathcal{F},\mu)$, stable under products
  (i.e. $\mathcal{A}$ is an algebra) and under the action of
  $\mathcal{C}^{\infty}$-functions $\Psi \colon \R^k \to \R$.
\item $\mathcal{A}_0 \subseteq \mathcal{A}$ is a subalgebra of $\mathcal{A}$ consisting
  of bounded functions which are dense in $L^p(E,\mu)$ for every $p \in [1,\infty)$. We assume
  that $\mathcal{A}_0$ is also stable under the action of smooth functions $\Psi$
  as above and also that $\mathcal{A}_0$ is an ideal in $A$ (if $F \in
  \mathcal{A}_0$ and $G \in \mathcal{A}$, then $FG \in \mathcal{A}_0$) 
\item The carré du champ operator $\Gamma \colon \mathcal{A}_0 \times \mathcal{A}_0 \to
  \mathcal{A}_0$ is a bilinear symmetric map such that $\Gamma(F,F) \geq 0$ for all $F \in
  \mathcal{A}_0$. For every $F \in \mathcal{A}_0$ there exists a finite constant $c_F$ such
  that for every $G \in \mathcal{A}_{0}$
  \begin{equation*}
    \abs{\int_E^{} \Gamma(F,G) \diff{\mu}} \leq c_F \norm{G}_2,
  \end{equation*}
  where $\norm{G}_2^2 = \int_E^{}G^2 \diff{\mu}$.
  The Dirichlet form $\mathcal{E}$ is defined on $\mathcal{A}_0 \times \mathcal{A}_0$ by
  \begin{equation*}
    \mathcal{E}(F,G) = \int_E^{} \Gamma(F,G) \diff{\mu}.
  \end{equation*}
  \item
  The domain $\mathcal{D}(\mathcal{E}) \subseteq L^2(E,\mu)$ is obtained by completing
  $\mathcal{A}_0$ with respect to the norm $\norm{F}_{\mathcal{E}} = \left(
    \norm{F}_2 + \mathcal{E}(f,f) \right)^{1/2}$. The Dirichlet form 
  $\mathcal{E}$ and the carré du champ operator $\Gamma$ are extended to
  $\mathcal{D}(\mathcal{E}) \times \mathcal{D}(\mathcal{E})$ by continuity and
  polarization. We thus have that $\Gamma \colon \mathcal{D}(\mathcal{E}) \times
  \mathcal{D}(\mathcal{E}) \to L^1(E,\mu)$.
\item $L$ is a linear operator, defined on $\mathcal{A}_0$
  via the \emph{integration by parts formula}
  \begin{equation}
    \label{eq:21}
    \int_E^{} G L F \diff{\mu} =  - \int_E^{} \Gamma(F,G) \diff{\mu}
  \end{equation}
  for all $F,G \in \mathcal{A}_0$. We assume that $L(\mathcal{A}_0) \subseteq
  \mathcal{A}_0$. 
\item The domain $\mathcal{D}(L) \subseteq \mathcal{D}(\mathcal{E})$ 
  consists of all $F \in \mathcal{D}(\mathcal{E})$ such that
  \begin{equation*}
    \abs{\mathcal{E}(F,G)} \leq c_F \norm{G}_2
  \end{equation*}
  for all $G \in \mathcal{D}(\mathcal{E})$, where $c_F$ is a finite constant. The
  operator $L$ is extended from $\mathcal{A}_0$ to $\mathcal{D}(L)$ by the
  integration by parts formula~\eqref{eq:21}. On $\mathcal{D}(L)$, $L$ is by
  construction self-adjoint (as $\Gamma$ is symmetric). We assume that $L 1 = 0$ and
  that $L$ is ergodic: $LF =0$ implies that $F$ is constant for all $F \in \mathcal{D}(L)$.
\item The operator $L \colon \mathcal{A} \to \mathcal{A}$ is an extension of $L
  \colon \mathcal{A}_0 \to \mathcal{A}_0$. On $\mathcal{A} \times \mathcal{A}$, the
  carré du champ $\Gamma$ is defined by
  \begin{equation}
    \label{eq:42}
    \Gamma(F,G) = \frac{1}{2} \left( L(FG) - FLG - GLF \right).
  \end{equation}
\item For all $F \in \mathcal{A}$, we assume $\Gamma(F,F) \geq 0$ with equality if, and only
  if, $F$ is constant.
\item The \emph{diffusion property} holds. For all $\mathcal{C}^{\infty}$-functions
  $\Psi \colon \R^p \to \R$ and $F_1,\dots,F_p,G \in \mathcal{A}$ one has
  \begin{equation}
    \label{eq:24}
    \Gamma \left( \Psi \left( F_1,\dots,F_p \right),G \right)
    =
    \sum_{j=1}^p \partial_j \Psi(F_1,\dots,F_p) \, \Gamma(F_j,G)
  \end{equation}
  and
  \begin{equation}
    \label{eq:30}
    L \Psi(F_1,\dots,F_p) = \sum_{i=1}^p \partial_i \Psi(F_1,\dots,F_p) L F_{i}
    +
    \sum_{i,j=1}^p \partial_{ij} \Psi(F_1,\dots,F_p) \, \Gamma(F_i,F_j).
  \end{equation}
\item The integration by parts formula~\eqref{eq:21} continues to hold if $F \in
  \mathcal{A}$ and $G \in \mathcal{A}_0$ (or vice versa).
\end{enumerate}
Of course, one can also introduce a symmetric Markov semigroup and associated
Markov process with infinitesimal generator $L$ defined on its domain
$\mathcal{D}(L)$ but as we will not make direct use of both of these objects in
this paper, we again refer to~\cite[Part I, Chapter
3]{bakry_analysis_2014} for deails.

To summarize, we have an algebra $\mathcal{A}$ of random variables on some
probability space $(E,\mathcal{F},\mu)$ on which the carré du champ operator $\Gamma$
and the generator $L$ act. The measure $\mu$ is called the invariant measure of
$L$. Note that there is no integrability assumption on the
elements of $\mathcal{A}$. The $L^2 \left( E,\mu \right)$-domains of $L$ and $\Gamma$ are denoted by
$\mathcal{D}(L)$ and $\mathcal{D}(\mathcal{E}) \times \mathcal{D}(\mathcal{E})$,
respectively, and both, $\mathcal{D}(L)$ and $\mathcal{D}(\mathcal{E})$
are dense in $L^2(E,\mu)$. By construction, one has $\mathcal{A}_0 \subseteq \mathcal{D}(L) \subseteq
\mathcal{D}(\mathcal{E}) \subseteq L^2(E,\mu) \subseteq \mathcal{A}$. 

A model example of the setting described above is the Markov triple 
$(\R^d,\gamma_d,\Gamma)$, where $\gamma_d$ is the $d$-dimensional Gaussian measure and $\Gamma =
\left\langle \nabla f,\nabla f \right\rangle_{\R^d}$ the carré du champ of the $d$-dimensional
Ornstein-Uhlenbeck generator $L$ given by $Lf= x \cdot \nabla f + \Delta f$. A suitable
algebra $\mathcal{A}$ is given by polynomials in $d$ variables. In infinite
dimension, one obtains the infinite-dimensional Ornstein-Uhlenbeck generator on
Wiener space with Wiener measure as invariant distribution. In this case, we  have that $L=-\delta D$, where
$\delta$ is the Malliavin divergence operator (also called Skorohod integral) and
$D$ the Malliavin derivation operator with carré du champ operator $\Gamma$
given by $\Gamma(F,G)=\left\langle DF,DG \right\rangle_{\mathfrak{H}}$, where $\mathfrak{H}$
denotes the underlying Hilbert space. For further details on this example, see
\cite{bouleau_dirichlet_1991,nualart_malliavin_2006}, or
\cite{nourdin_normal_2012}.

The Ornstein-Uhlenbeck generator is a particular example of Pearson generators
which will be discussed in detail in Section~\ref{secchaosstruct}. General
references with more examples fitting our framework are \cite{bakry_diffusions_1985,bakry_symmetric_2014,bakry_analysis_2014,bakry_analysis_2014,fukushima_dirichlet_2011}.

\subsection{Stein's method for invariant measures of diffusions}
\label{steinfordiff}
In this section, we present Stein's method for invariant measures of
diffusions. Note that if $\mu$ is a measure which is absolutely continuous with
respect to the Lebesgue measure and admits a density $p$ as well as a second
moment, then under very minimal assumptions there exists a Markov diffusion
generator $L$ having $\mu$ as its invariant measure. To be more precise, let
$\mu$ be a probability measure admitting a density $p$ with support $(l,u) \subseteq \R$, $-\infty \leq
l < u \leq +\infty$, and denote $m = \int_{\R}^{} x
p(x)dx$ and
\begin{equation}
  \label{eq:32}
  \sigma^2(x) =
  \frac{-2\theta \int_{-\infty}^x (y-m) p(y) \diff{y}}{p(x)},\ x \in (l,u).
\end{equation}
Then, the stochastic differential equation
\begin{equation}
  \label{eq:28}
  \diff{X_t}
  =
  -\theta(X_t-m) \diff{t}
  +
  \sigma(X_t) \diff{B_t},\ X_t \in (l,u),
\end{equation}
where $\left\{ B_t \colon t\geq 0 \right\}$ is a Brownian motion, has a unique weak Markovian solution with
invariant measure $\mu$ (see \cite{bibby_diffusion-type_2005}, Theorem
2.3). The support of the density $p$ could very well be taken to
  be a union of open intervals, but we treat
  here the case of one open interval in order not to make the notation
  heavier than it needs to be.

Stein's density approach (see \cite{stein_approximate_1986} for a
detailed treatment) allows us to
characterize the invariant measure $\mu$ of the diffusion
\eqref{eq:28} through the following theorem, called Stein's lemma for
invariant measures of diffusions (see
 \cite[Proposition 6.4]{nourdin_steins_2009} or
 \cite[Lemma 6]{eden_nourdin-peccati_2015}).

\begin{theorem}
  \label{thm:3}
  Let
$\mu$ be a probability measure admitting a density $p$ with support $(l,u) \subseteq \R$, $-\infty \leq
l < u \leq +\infty$, such that $\int_{\R}^{}\abs{x}p(x)dx < \infty$ and
$\int_{\R}^{}xp(x)dx =m$. Define the function
\begin{equation*}
\tau(x)=\frac{1}{2}\sigma^2(x)\mathds{1}_{(l,u)}(x),\quad x \in \R,
\end{equation*}
where $\sigma^2$ is defined in terms
of $p$ by~\eqref{eq:32} and let $Z$ be a random variable having
distribution $\mu$.
\begin{enumerate}[(i)]
\item For every differentiable $\varphi$ such that
  $\tau(Z)\varphi'(Z) \in L^1(\Omega)$, one has that
  $(Z-m)\varphi(Z) \in L^1(\Omega)$ and 
\begin{equation*}
\E \left(\tau(Z)\varphi'(Z) - \theta(Z-m)\varphi(Z) \right) = 0.
\end{equation*}
\item Let $X$ be a real-valued random variable with an absolutely
  continuous distribution. If for every differentiable $\varphi$
  such that $\tau(X)\varphi'(X) \in L^1(\Omega)$ and $(X-m)\varphi(X) \in L^1(\Omega)$, one has
  that 
  \begin{equation}
    \label{eq:33}
\E \left(\tau(X)\varphi'(X) - \theta(X-m)\varphi(X) \right)=0,
\end{equation}
then $X$ has distribution $\mu$.
\end{enumerate}
\end{theorem}
Based on the above Stein lemma, one can use the now well established
Stein methodology to quantitatively measure the distance between the
law of a random variable $X$ and the law of a random variable $Z$
corresponding to an invariant measure of a
diffusion. The generalization of the original Stein method to
invariant measures of diffusions has been recently studied in \cite{kusuoka_steins_2012}
and further developed in \cite{eden_nourdin-peccati_2015}. In order to present
this method, we need to introduce separating classes of functions and
probabilistic distances.
\begin{definition}
  \label{def:2}
Let $\mathscr{H}$ be a collection of Borel-measurable functions on
$\R$. We say that the class $\mathscr{H}$ is \emph{separating} if the
following property holds: any two real-valued random variables $X, Y$
verifying $h(X), h(Y) \in L^1(\Omega)$ and $\E(h(X)) = \E(h(Y))$ for
every $h \in \mathscr{H}$, are necessarily such that $X$ and $Y$ have
the same distribution.
\end{definition}
Separating classes of functions can be used to introduce distances between
probability measures in the following way.
\begin{definition}
  Let $\mathscr{H}$ be a separating class in the sense of Definition~\ref{def:2}
  and let $X,Y$ be real-valued random variables such that $h(X), h(Y) \in L^1(\Omega)$
  for every $h \in \mathscr{H}$. Then the \emph{distance} $d_{\mathscr{H}}(X,Y)$
  between the distributions $X$ and $Y$ is given by
\begin{equation}
  \label{distancegeneralform}
d_{\mathscr{H}}(X,Y) = \sup_{h \in \mathscr{H}}\abs{\E(h(X)) - \E (h(Y))}.
\end{equation}
\end{definition}
One can show that $d_{\mathscr{H}}$ is a metric on some subset of the class of
all probability measures on $\R$ (see \cite[Chapter 11]{dudley_real_2002}). With some abuse of
language, one often speaks of the ``distance between random variables'' when really the distance between the laws of these random variables
is meant. We will call a given distance $d_{\mathscr{H}}$ \emph{admissible} for
a set $\mathcal{M}$ of random variables if $d_{\mathscr{H}}(X,Y)$ is well
defined for all $X,Y \in \mathcal{M}$, i.e. if it holds that $\Ex{h(X)}<\infty$ for 
all $X \in \mathcal{M}$ and $h \in \mathscr{H}$.
As an example of a distance as introduced above, one can take
$\mathscr{H}$ to be the class of Lipschitz continuous and bounded
functions. This yields the well-known Fortet-Mourier (or bounded Wasserstein)
distance denoted by $d_{FM}$, which metrizes convergence in distribution and is
defined for \emph{all} real-valued random variables. It is therefore admissible
for any set $\mathcal{M}$ of random variables. Other
distances (with smaller domains) are the total variation, Kolmogorov or
Wasserstein distance (see REF book Giovanni/Ivan Appendix).
A Stein equation is an ordinary differential equation linking the
notion of distance (through the right-hand side of \eqref{distancegeneralform}) to the characterizing expression of a distribution
appearing in Stein's lemma (the right-hand side of \eqref{eq:33} for
instance). More precisely, a Stein equation associated to the Stein
characterization \eqref{eq:33} is given by 
\begin{equation}
  \label{steineqinvmeasureofdiff}
\tau(x)f'(x) - \theta(x-m)f(x) = h(x) - \E(h(Z)),
\end{equation}
where $Z$ is a random variable with distribution $\mu$ given by the
invariant measure of the diffusion in \eqref{eq:28}. It is
straightforward to check that this equation has
a continuous solution on $\R$ for each $h \in \mathscr{H}$, denoted by $f_h$, and given
by 
\begin{align*}
  f_h(x) &=  \frac{1}{\tau(x)p(x)}\int_l^x(h(y) - \E(h(Z)))p(y)dy \\
  &= -\frac{1}{\tau(x)p(x)}\int_x^u(h(y) - \E(h(Z)))p(y)dy
\end{align*}
for $x \in (l,u)$, and by
\begin{equation*}
f_h(x) = -\frac{h(x)-\E(h(Z))}{\theta(x-m)}
\end{equation*}
when $x \notin (l,u)$ (as in that case, $\tau(x) = 0$). Now, let $X$ be a real-valued random variable with an
absolutely continuous distribution. By letting $x =
X$ in \eqref{steineqinvmeasureofdiff}, taking expectations and the supremum over
the separating class of test functions $\mathscr{H}$ on both sides, we
can express the distance $d_{\mathscr{H}}$ in
\eqref{distancegeneralform} as 
\begin{equation}
  \label{distanceinsteinsmethodexpression}
d_{\mathscr{H}}(X,Z) = \sup_{h \in \mathscr{H}} \abs{\E
  \left(\tau(X)f_h'(X)\right) -\E \left( \theta(X-m)f_h(X)\right)}.
\end{equation}

The following result, a proof of which can be found in \cite{eden_nourdin-peccati_2015}, combines results from
\cite{kusuoka_steins_2012} and \cite{eden_nourdin-peccati_2015} and provides sufficient conditions under which useful
estimates for $f_h'$ can be obtained. 
\begin{lemma}
  \label{lemmaedenviquez}
Let the function $\sigma^2$, associated to a density $p$ with support $(l,u) \subseteq \R$, $-\infty \leq
l < u \leq +\infty$, be given by \eqref{eq:32}. If $u = \infty$, then
assume that $\varliminf_{x \to u}\sigma^2(x) > 0$, and if $l =
-\infty$, assume that $\varliminf_{x \to l}\sigma^2(x) >
0$. Furthermore, suppose that there exists a positive function $g \in
\mathcal{C}^1((l,u),\R_{+})$ such that
\begin{enumerate}[(i)]
\item $0< \varliminf_{x \to u}\sigma^2(x)/g(x) \leq \varlimsup_{x \to
    u}\sigma^2(x)/g(x) < \infty$;
\item $\lim_{x \to u}g'(x) \in \left[ -\infty,+\infty \right]$;
  \item $0< \varliminf_{x \to l}\sigma^2(x)/g(x) \leq \varlimsup_{x \to
    l}\sigma^2(x)/g(x) < \infty$;
  \item $\lim_{x \to l}g'(x) \in \left[ -\infty,+\infty \right]$.
\end{enumerate}
Then the solution $f_h$ to the Stein equation
\eqref{steineqinvmeasureofdiff}, for a given test function $h \in
\mathscr{H}$ such that $\norm{h'}_{\infty} < \infty$, satisfies 
\begin{equation*}
\norm{f_h'}_{\infty} \leq k \norm{h'}_{\infty},
\end{equation*}
where the constant $k$ does not depend on $h$.
\end{lemma}

\subsection{Pearson diffusions}
\label{pearsondiffsubsec}

Pearson diffusions are It\=o diffusions with mean
reverting linear drift whose squared diffusion coefficient is a quadratic
polynomial, i.e., a stationary solution of the stochastic differential equation 
\begin{equation}
  \label{eq:29}
  \diff{X_t} = - \theta (X_t-m) \diff{t} + \sqrt{2\theta b(X_t)} \diff{B_t},
\end{equation}
where
\begin{equation*}
b(x)=b_2x^2+b_1x+b_0.
\end{equation*}
Here, $m,b_2,b_1,b_0$ are real constants, $\theta >0$ determines the speed of
mean reversion and $m$ is the stationary mean. Recall that the scale and speed
densities $s$ and $p$, respectively, are defined as 
\begin{equation*}
  s(x) = \exp \left(-2 \int_{x_0}^x \frac{a(u)}{\sigma^2(u)} \diff{u} \right)
  \qquad \text{ and } \qquad
  p(x) = \frac{1}{s(x) \sigma^2(x)}.
\end{equation*}
In our case, we have the relation
\begin{equation}
  \label{eq:6}
  p'(x) = - \frac{(2b_2+1)x -m+b_1}{b_2x^2+b_1x+b_0} \, p(x),
\end{equation}
which was originally used by Pearson (see \cite[page 360]{pearson_contributions_1895}) to introduce these distributions. From~\eqref{eq:6} one also sees that the class of Pearson diffusions is
closed under linear transformations. Explicitly, if $X_t$ satisfies the
stochastic differential equation~\eqref{eq:29}, then $\widetilde{X}_t = \gamma X_t 
+ \delta$ satisfies
\begin{equation*}
  \diff{\widetilde{X}_t} = \widetilde{a}(\widetilde{X}_t) \diff{t} +
  \widetilde{\sigma}(\widetilde{X}_t) \diff{B_t}, 
\end{equation*}
where $\widetilde{a}(x)=\theta \left( x - \gamma m + \delta \right)$ and
\begin{equation*}
  \widetilde{\sigma}^2(x) = 2 \theta \left( b_2x^2+\left( b_1\gamma-2b_2\delta
    \right)x + b_0\gamma^2-b_1\gamma\delta+b_2\delta^2 \right).
\end{equation*}
Up to such linear transformations, Pearson diffusions can be categorized into
the six classes listed below together with their invariant distributions,
densities, means and diffusion coefficients. A detailed analysis and
classification of Pearson diffusions can for example be found in
\cite{johnson_continuous_1994,johnson_continuous_1995,forman_pearson_2008}.

\begin{enumerate}[1.]
\item Gaussian distribution with parameters $m \in \R$ and $\sigma
  >0$. It has state space $\R$, mean $m$, as well as density function and diffusion coefficients given by   
\begin{align*}
p(x) \propto e^{-\frac{(x-m)^2}{2\sigma^2}}, \qquad b(x) = \sigma^2.
\end{align*}
The Gaussian distribution has moments of all orders.
\item Gamma distribution with parameters $\alpha, \beta >0$. It has
  state space $(0,\infty)$, mean $\frac{\alpha}{\beta}$, as well as density function and diffusion coefficients given by   
\begin{align*}
p(x) \propto x^{\alpha -1}e^{-\beta x}, \qquad b(x) = \frac{x}{\beta}.
\end{align*}
The Gamma distribution has moments of all orders.
\item Beta distribution with parameters $\alpha, \beta >0$. It has
  state space $(0,1)$, mean $\frac{\alpha}{\alpha+\beta}$, as well as density function and diffusion coefficients given by   
\begin{align*}
p(x) \propto x^{\alpha -1}(1-x)^{\beta-1}, \qquad b(x) = -\frac{x^2}{\alpha + \beta} + \frac{x}{\alpha + \beta}.
\end{align*}
The Beta distribution has moments of all orders.
\item Skew $t$-distribution with parameters $m, \nu, \lambda \in \R$,
  $\alpha>0$. It has state space $\R$, mean $\frac{(2m-1)\lambda + \alpha \nu}{2(m-1)}$, as well as density function and diffusion coefficients given by   
\begin{align*}
p(x) &\propto \left( 1+\left(
    \frac{x-\lambda}{\alpha}\right)^2\right)^{-m}e^{-\nu \arctan\left(
    \frac{x-\lambda}{\alpha}\right)},\\
 b(x) &= \frac{x^2}{2(m-1)} - \frac{\lambda x}{2(m-1)} + \frac{\lambda^2 + \alpha^2}{2(m-1)}.
\end{align*}
The skew $t$-distribution has moments of order $p$ for $p < 2m-1$.
\item Inverse gamma distribution with parameters $\alpha, \beta >0$. It has
  state space $(0,\infty)$, mean $\frac{\beta}{\alpha-1}$, as well as density function and diffusion coefficients given by   
\begin{align*}
p(x) \propto x^{-(\alpha-1)}e^{-\frac{\beta}{x}}, \qquad b(x) = \frac{x^2}{\alpha -1}.
\end{align*}
The inverse gamma distribution has moments of order $p$ for $p < \alpha$.
\item $F$-distribution with parameters $d_1, d_2 >0$. It has
  state space $(0,\infty)$, mean $\frac{d_2}{d_2 - 2}$, as well as density function and diffusion coefficients given by   
\begin{align*}
p(x) \propto x^{\frac{d_1}{2}-1}\left(1+\frac{d_1}{d_2}x \right)^{-\frac{d_1+d_2}{2}}, \qquad b(x) = \frac{2x^2}{d_2 -2} + \frac{2d_2 x}{d_1(d_2 -2)}.
\end{align*}
The $F$-distribution has moments of order $p$ for $p < \frac{d_2}{2}$.
\end{enumerate}
~\\
Pearson diffusions are particular (one-dimensional) examples fitting the Markov
triple structure introduced in Subsection \ref{s-10}. 
The generator $L$ acts on $L^2(E,\mu)$ via
\begin{equation*}
  L f(x) = -(x-m) f'(x) + b(x) f''(x),
\end{equation*}
where $b$ is the quadratic polynomial appearing in \eqref{eq:29}. Its
invariant measure $\mu$ is a Pearson distribution and it is
furthermore symmetric, ergodic and diffusive (in the sense of
\eqref{eq:30}). The set $\Lambda$ of eigenvalues of $L$ is given by
\begin{equation}
  \label{eq:19}
  \Lambda = \left\{ -n \left( 1 - (n-1)b_2 \right)\theta \, \colon \, n \in \N_0, \, b_2 < \frac{1}{2n-1} \right\}
\end{equation}
and the corresponding eigenfunctions are the well-known orthogonal polynomials
associated with the respective laws (Hermite, Laguerre and Jacobi polynomials for the Gaussian,
Gamma and Beta distributions, respectively, and Romanovski-Routh, Romanovski-Bessel and Romanovski-Jacobi polynomials for the skew $t$-, inverse gamma and $F$-distributions.
From formula \eqref{eq:19}, we see that polynomials up to
degree $n=\lfloor\frac{1+b_2}{2b_2}\rfloor$ are (square integrable) eigenfunctions, so
that $\mu$ has moments up to order $2n = \lfloor  1
+ 1/b_2 \rfloor$. 
Note that the cardinality of $\Lambda$ is infinite if $b_2\leq0$ and
finite if $b_2>0$. Consistent with the general theory of Markov
generators presented in Subsection \ref{s-10}, zero is always contained
in $\Lambda$ and all other eigenvalues are negative.
\\~\\
The structure of the spectrum $S$ of such a Pearson generator can thus be
described as follows.
\begin{enumerate}[(i)]
\item If $\mu$ is a Gaussian, Gamma or Beta distribution, then $S$ is purely discrete and  consists of infinitely many eigenvalues, each of multiplicity one. These eigenvalues are the negative integers including zero in the Gaussian and Gamma case. Eigenfunctions are the associated orthogonal polynomials (Hermite, Laguerre or Jacobi).
\item If $\mu$ is a skew $t$-, inverse Gamma or scaled $F$-distribution, then $S$
  contains a discrete and a continuous part. The discrete part consists of only finitely many eigenvalues.
\end{enumerate}
For later reference, we note that for a Pearson distribution $\mu$, the Stein
characterization~\eqref{eq:33} in Theorem~\ref{thm:3} becomes
\begin{equation}
      \label{eq:5}
      \Ex{ b(X)\mathds{1}_{(l,u)}(X) \varphi'(X) - (X-m) \varphi(X)} = 0,
\end{equation}
where again $b(x)=b_2x^2+b_1x+b_0$ is the associated quadratic
polynomial. Identity~\eqref{eq:5} gives a recursion formula for computing
the moments of a given Pearson distribution. Indeed, if the law of $X$ is a
Pearson distribution  with moments up to order $p+2$, then~\eqref{eq:5} with $\varphi(x)=x^{p+1}$
reads 
\begin{equation*}
(p+1) \Ex{b(X) X^p} - \Ex{(X-m)X^{p+1}}
  = 0.
\end{equation*}
This yields
\begin{equation*}
  \left( b_2(p+1)-1\right) \Ex{X^{p+2}} + \left( b_1(p+1)+m \right) \Ex{X^{p+1}}
  +  (p+1) b_0 \Ex{X^p} = 0
\end{equation*}
with $\Ex{X}=m$. Recall from the previous paragraph that the condition for the existence of moments of order $p$
is $p < 1 + b_2^{-1}$, so that four moments exist if and only if $b_2 <
\frac{1}{3}$. In this case, we start with with $\Ex{X}=m$ and get
\begin{align*}
  \Ex{X^2}
  &=
    \frac{(b_1+m)m+b_0}{1-b_2},
  \\
  \Ex{X^3}
  &=
    \frac{(2b_1+m)\left( \left( b_1+m \right)m+b_0 \right)}{(1-b_2)(1-2b_2)}
    +
    \frac{2b_0m}{1-2b_2},
  \\
  \Ex{X^4}
  &=
    \frac{(3b_1+m)(2b_1+m)\left( \left( b_1+m \right)m+b_0
    \right)}{(1-b_2)(1-2b_2)(1-3b_2)}
    +
    \frac{(3b_1+m)2b_0m}{(1-2b_2)(1-3b_2)}
  \\
  &\phantom{=1}
    +
    \frac{3b_0 \left( \left( b_1+m \right)m+b_0 \right) }{1-3b_2}.
\end{align*}
For further analysis of the spectrum of such Pearson generators and
general motivation on studying Pearson diffusions, see \cite{avram_spectral_2013}.
\section{Main results}
\label{mainresultssec}

Throughout this section, we always work in the Markov setting introduced in
Subsection~\ref{s-10}. We thus have a probability space
$(E,\mathcal{F},\mu)$ and the two operators $L$ and $\Gamma$ with
their $L^2$-domains $\mathcal{D}(L)$ and $\mathcal{D}(\mathcal{E})\times
\mathcal{D}(\mathcal{E})$ respectively, where $\mathcal{D}(L)
\subseteq \mathcal{D}(\mathcal{E}) \subseteq L^2(E,\mu)$. As is customary in this context, we continue to use the integral
notation for mathematical expectation, so that for example the expectation of a
random variable $G \in L^1(E,\mu)$ is denoted by $\int_E^{} G \diff{\mu}$.

\subsection{Carr\'e du champ characterization}
\label{s-2}

As a first result, we show how the Stein characterization~\eqref{eq:33} can
naturally be translated into a condition involving the carré du champ operator
$\Gamma$. To do so, we need the pseudo-inverse $L^{-1}$ of $L$,
satisfying for any
$X \in D(L)$,
\begin{equation}
  \label{eq:43}
  LL^{-1} X = L^{-1}L X = X - \pi_0(X),
\end{equation}
where $\pi_0(X)=\int_E^{} X \diff{\mu}$ denotes the orthogonal projection
of $X$ onto $\on{ker}(L)$ (recall that the kernel of $L$ by assumption only
consists of constants). For completeness, we recall how this
pseudo-inverse is constructed.  By self-adjointness of $L$, considered as an operator on $\mathcal{D}(L)$,
we have that $\mathcal{D}(L) = \on{ker}(L) \oplus (\on{ran}(L) \cap
\mathcal{D}(L))$. Therefore, we can define $L^{-1}$ on $\on{ran}(L) \cap
\mathcal{D}(L)$ (as $L$ is injective there) and then extend it to
$\mathcal{D}(L)$ by setting $L^{-1}X=0$ if $X \in
\on{ker}(L)$. 
\\~\\
We are now ready to state the announced
carr\'e du champ characterization.

\begin{theorem}
  \label{thm:5}
  Let
$\nu$ be a probability measure admitting a density $p$ with support $(l,u) \subseteq \R$, $-\infty \leq
l < u \leq +\infty$, such that $\int_{\R}^{}\abs{x}p(x)dx < \infty$ and
$\int_{\R}^{}xp(x)dx =m$. Define the function
$\tau(x)=\frac{1}{2}\sigma^2(x)\mathds{1}_{(l,u)}(x)$, $x \in \R$, where $\sigma^2$ is defined in terms
of $p$ by~\eqref{eq:32}. Let $G \in \mathcal{D}(L)$ with an
absolutely continuous distribution and mean $m$. Then $G$ has distribution $\nu$ if, and only if,
\begin{equation*}
\Gamma(G,-L^{-1}G) = \theta^{-1}\tau(G)
\end{equation*}
almost surely. 
\end{theorem}
\begin{proof}
Let $\varphi \in \mathcal{C}^{\infty}(\R,\R)$ be such that $\tau(G)\varphi'(G) \in 
L^1(\Omega)$ and $(G-m)\varphi(G) \in L^1(\Omega)$. It suffices to prove that
\begin{equation}
  \label{eq:35}
  \int_E^{} (G-m) \varphi(G) \diff{\mu} = \int_E^{}  \varphi'(G) \Gamma(G,-L^{-1}G) \diff{\mu},
\end{equation}
as this implies
\begin{equation}
  \label{eq:31}
  \int_E^{} \tau(G) \varphi'(G) - \theta (G-m) \varphi(G) \diff{\mu}
  =
  \theta
  \int_E^{} \varphi'(G) \left( \theta^{-1} \tau(G) - \Gamma(G,-L^{-1}G) \right) \diff{\mu},
\end{equation}
so that the assertion follows from Theorem~\ref{thm:3}.
By the diffusion property, we have that
\begin{equation*}
  \Gamma(\varphi (G), -L^{-1}G) = \varphi'(G) \Gamma(F,-L^{-1}G).
\end{equation*}
Taking integrals, then applying the integration by parts formula
\eqref{eq:21}, and afterwards making use of the identity~\eqref{eq:43} yields
\begin{align*}
  \int_E^{} \varphi'(G) \Gamma(F,-L^{-1}G) \diff{\mu}
  &=
    \int_E^{} \Gamma(\varphi(G),-L^{-1}G) \diff{\mu}
  \\ &=
       \int_E^{} \varphi(G) LL^{-1} G \diff{\mu}
       \\ &= \int_E^{} \varphi(G) (G-m) \diff{\mu}.
\end{align*}
\end{proof}

Using Stein's method, we obtain the following quantitative version of Theorem~\ref{thm:5}.

\begin{theorem}
\label{thm:4}
  Let $\nu$ be a measure with density $p$ and let $\sigma^2$ be given
  by~\eqref{eq:32}. Assume that $\sigma^2$ satisfies the assumptions
  of Lemma \ref{lemmaedenviquez}, and let $\tau(x) =
  \frac{1}{2}\sigma^2(x)\mathds{1}_{(l,u)}(x)$, $x \in \R$. Furthermore, let $G \in \mathcal{D}(L)$ such that $\int_E^{}
  \tau(G) \diff{\mu} < \infty$ and $\int_E^{} G \diff{\mu} =
  m$. Finally, let $Z$ be a random variable with distribution
  $\nu$. Then one has
  \begin{equation}
    \label{eq:34}
d_{\mathscr{H}}(G,Z) \leq c_{\mathscr{H}}
\int_E^{}\abs{\Gamma(G,-L^{-1}G) - \theta^{-1}\tau(G)}d\mu,
\end{equation}
where $d_{\mathscr{H}}$ is an admissible distance for $G$ and $Z$, defined via~\eqref{distancegeneralform} using a separating class $\mathscr{H}$ of absolutely continuous test functions such
that $\sup_{h \in \mathscr{H}}\norm{h'}_{\infty} < \infty$ and $c_{\mathscr{H}}$ is a
positive constant depending solely on the class $\mathscr{H}$.
\end{theorem}
\begin{remark}
  Note that the Fortet-Mourier metric always satisfies the assumptions of
  Theorem~\ref{thm:4}. In concrete situations, when both the law $\nu$ and the
  generator $L$ are known, one can often take stronger distances such as
  Kolmogorov or total variation.
\end{remark}

\begin{proof}[Proof of Theorem \ref{thm:4}]
On the one hand, by using Stein's method for invariant measures of
diffusions (see Subsection \ref{steinfordiff}), we can write, using \eqref{distanceinsteinsmethodexpression}, 
\begin{equation}
  \label{expfromsteinmethod}
d_{\mathscr{H}}(G,Z) = \sup_{h \in \mathscr{H}} \abs{\int_E^{}\tau(G)f_h'(G)d\mu -\int_E^{} \theta(G-m)f_h(G)d\mu},
\end{equation}
where $f_h$ denotes the solution to the Stein equation
\eqref{steineqinvmeasureofdiff}. On the other hand, proceeding as in the proof
of Theorem~\ref{thm:3}, we have
  \begin{equation*}
    \int_E^{} (G-m) f_h(G) \diff{\mu}
=
      \int_E^{} f_h'(G) \Gamma(G,-L^{-1}G) \diff{\mu}.
    \end{equation*}
    Plugged into~\eqref{expfromsteinmethod} and applying the Hölder inequality,
    we obtain
  \begin{align*}
    d_{\mathscr{H}}(G,Z) &= \sup_{h \in \mathscr{H}}
                           \abs{\int_E^{}\left(\tau(G)f_h'(G) -\theta f_h'(G) \Gamma(G,-L^{-1}G)\right) \diff{\mu}}\\
                           &\leq \sup_{h \in \mathscr{H}} \norm{f_h'}_{\infty} \theta
    \int_E^{} \abs{
     \Gamma(G,-L^{-1}G) - \theta^{-1}\tau(G)
    }
    \diff{\mu},
\end{align*}
so that the assertion follows by Lemma~\ref{lemmaedenviquez} with
$c_{\mathscr{H}} = k \theta\sup_{h \in \mathscr{H}} \norm{h'}_{\infty} < \infty $.
\end{proof}

\subsection{Markov chaos and Four Moments Theorems}
\label{s-1}
This subsection introduces the concept of chaotic eigenfunctions, for
which the general bound obtained in Theorem \ref{thm:4} can further be
bounded by a
finite linear combination of moments. Chaotic eigenfunctions have
first been introduced in \cite{ledoux_chaos_2012} and
a more general definition has been given in \cite{azmoodeh_fourth_2014}. In
order to also be able to deal with heavy-tailed invariant measures, we have to
extend this definition once again by introducing the new notion of chaos grade.

We continue to assume as given a Markov structure as introduced in
Subsection~\ref{s-10} and denote the spectrum of the generator $L$ (defined on
$\mathcal{D}(L)$) by $S$. As $-L$ is non-negative and symmetric,
one has $S \subseteq (-\infty,0]$. Let $\Lambda \subseteq  S$ denote the set of eigenvalues of $L$. We
always have that $0 \in \Lambda$ as by assumption $L1=0$. Chaotic random variables are then defined as follows.

\begin{definition}
  \label{def:1}
An eigenfunction $F$ with respect to an eigenvalue $-\lambda$ of $L$ is called \textit{chaotic}, if there exists $\eta>1$ such that
  $-\eta \lambda$ is an eigenvalue of $L$ and
  \begin{equation}
    \label{eq:10}
    F^2 \in \bigoplus_{\substack{-\kappa \in \Lambda \\ \kappa \leq \eta \lambda}} \ker \left( L
      + \kappa \on{Id} \right).
  \end{equation}
  In this case, the smallest $\eta$ satisfying~\eqref{eq:19} is called the
  \emph{chaos grade} of $F$.
\end{definition}
In other words, an eigenfunction is called chaotic if its square can
be expressed as a sum of eigenfunctions.
\begin{remark}
  \begin{enumerate}[(i)]
  \item  As we assume that $L^2(E,\mathcal{F},\mu)$ is separable, the set $\Lambda$ and
    therefore the direct orthogonal sum~\eqref{eq:10} of eigenspaces is at most countable.
  \item The chaos grade is invariant under scaling of the generator, in the
    sense that if $F$ is a chaotic random variable of $L$ with chaos grade $\eta$,
    then the chaos grade of $F$ remains unchanged when viewed as a
    chaotic random variable of $\alpha L$ for any $\alpha \in \R$.
  \end{enumerate}
\end{remark}

Let us give some examples to illustrate the concept.

\begin{example}
  \begin{enumerate}[1.]
  \item A model example is the generator of a Pearson distribution and we will
    study this example in detail in Section \ref{secchaosstruct}. At this point, let us briefly
    illustrate\ the chaos grade concept by treating the concrete case of the
    Gaussian distribution $\mu$. Here, the generator is the one-dimensional
    Ornstein-Uhlenbeck generator, acting on $L^2(\R,\mu)$. As is well known, the
    spectrum of $L$ consists of the negative integers and zero, which all are eigenvalues
    with the respective Hermite polynomials as eigenfunctions (the Hermite
    polynomial $H_p$ of order $p$ being an eigenfunction with respect to the
    eigenvalue $-p$). The square of such a Hermite polynomial $H_p$ can of
    course be expressed as a linear combination of Hermite polynomials up to
    order $2p$ and this expansion is given by the product formula
    \begin{equation*}
      H_p^2(x) = \sum_{j=1}^{p} c_{p,j} H_{2(p-j)}(x),
    \end{equation*}
    where $c_{p,j} = k! \binom{p}{j}^2$. Therefore, the chaos grade of $H_p$ is
    $2$.
  \item\label{item:1}
    The preceding example can also be looked at in infinite dimensions. Here, the
    one-dimensional Gaussian distribution is replaced with Wiener measure and
    $L$ is the infinite dimensional Ornstein-Uhlenbeck generator. The spectrum
    of $L$ still consists of the negative integers and zero, with the
    eigenfunctions now being  multiple Wiener-It\=o integrals of
  the form $F=I_p(f)$ (so that $LI_p(f) = -p I_p(F)$). The product formula for such integrals says that
  \begin{equation*}
    F^2 = I_p(f)^2 = \sum_{j=0}^{p} c_{p,j} I_{2(p-j)}(f_j),
  \end{equation*}
  where the constants $c_{p,j} $ are defined as in the previous example and the
  kernels $f_j$ are given in terms of so-called contractions of the original
  kernel $f$. This shows that any such multiple Wiener-It\=o integral is a chaotic
  eigenfunction in the sense of Definition~\ref{def:1} with chaos grade $2$.
\item Another example in dimension one is obtained by taking $L$ to be the Jacobi generator acting on
   $L^2([0,1],\nu)$, with invariant measure $\nu$ given by $\nu(\diff{x})= c_{\alpha,\beta}
   x^{\alpha-1}(1-x)^{\beta-1} 1_{[0,1]}(x) \diff{x}$ for some positive parameters
   $\alpha,\beta$. Then $L$ is such that
   \begin{equation*}
     Lf(x) = x(1-x) f''(x) + (\alpha - (\alpha+\beta)x) f'(x).
   \end{equation*}
   It is well known that the eigenvalues of $L$ are given by the
   Jacobi polynomials. The chaos grade of an eigenfunction associated
   to the eigenvalue $\lambda_n=-n \left( 1+\frac{n-1}{\alpha+\beta} \right)$ is given
   by $2 \left( 1+\frac{n}{n-1+\alpha+\beta} \right)$ (see Section
   \ref{secchaosstruct} for a full treatment of chaos grade
   characterizations). Note that the chaos grade in this case is no
   longer 2 and depends on the eigenvalue the eigenfunction is
   associated to. As in the Wiener case, a tensorization
   procedure (see Section \ref{secchaosstruct}) allows to generalize
   this example to higher dimensions.
  \end{enumerate}
  
\end{example}
\begin{remark}
For a systematic study of the chaos grades of eigenfunctions of
Pearson generators, see Section \ref{secchaosstruct}.
\end{remark}

We are now ready to prove Four Moments Theorems for Pearson
distributions. In all that follows, $F$ will denote an eigenfunction
of $L$, which is necessarily centered, and $G=F+m$ a translated version of $F$
which has then expectation $m \in \R$ as in the previous section. Furthermore, as
the six classes of Pearson diffusions given by~\eqref{eq:29} are 
invariant under linear transformations (see Section~\ref{pearsondiffsubsec}), we
assume from here on without loss of generality that $\theta = \frac{1}{2}$. 

\begin{theorem}
  \label{thm:1}
Let $\nu$ be a Pearson distribution associated to the diffusion given
  by~\eqref{eq:29} with mean $m$ and diffusion coefficient 
  $\sigma^2(x)=b(x)=b_2x^2+b_1x+b_0$, where $b_0, b_1, b_2 \in \R$.
  Let $F$ be a chaotic eigenfunction of $L$ with respect to the eigenvalue $-\lambda$, chaos grade $\eta$ and
  moments up to order $4$.  Set $G=F+m$. Then, if $\eta \leq 2(1-b_2)$, one has
  \begin{equation}
    \label{eq:23}
    \int_E^{} \left( \Gamma(G,-L^{-1}G) - b(G) \right)^2 \diff{\mu} \\\leq
    2
    \left( 1-b_2-\frac{\eta}{4}\right)\int_E^{}U(G)\diff{\mu},
  \end{equation}
whereas if $\eta > 2(1-b_2)$, one has

\begin{multline}
  \label{eq:12}
     \int_E^{} \left( \Gamma(G,-L^{-1}G) - b(G) \right)^2 \diff{\mu} \\\leq
    2 \left( 1-b_2-\frac{\eta}{4}\right)\int_E^{}U(G)\diff{\mu}+\frac{\xi(1-b_2)}{2}\int_E^{} Q^2(G)\diff{\mu},
  \end{multline}
  where
\begin{equation*}
\xi = \eta - 2(1-b_2)> 0,
\end{equation*}
and where the polynomials $Q$ and $U$ are
  given respectively by
  \begin{equation*}
Q(x) = x^2 + \frac{2(b_1+m)}{2b_2-1}x + \frac{1}{b_2-1}\left( b_0 + \frac{m(b_1+m)}{2b_2-1} \right),
\end{equation*}
and
\begin{equation*}
U(x) = (1-b_2)Q^2(x) - \frac{1}{12}(Q'(x))^3(x-m).
\end{equation*}
\end{theorem}
\begin{remark}\hfill
  \label{s-6}
  \begin{enumerate}[(i)]
    \item Observe that both $\int_E^{}U(G)\diff{\mu}$ and
      $\int_E^{}Q^2(G)\diff{\mu}$ are linear combinations of the first
      four moments of $G$, i.e. there exists coefficients
      $c_j,d_j$, $j = 0,\ldots,4$ such that
\begin{equation*}
\int_E^{}U(G)\diff{\mu} =\sum_{j=0}^{4}c_j \int_E^{}G^jd\mu \quad
\text{and}\quad \int_E^{}Q^2(G)\diff{\mu} =\sum_{j=0}^{4}d_j \int_E^{}G^jd\mu.
\end{equation*}
The coefficients $c_j,d_j$ only depend on the coefficients of the polynomial $b$
and the mean $m$ of the target distribution and for convenience are given in Table \ref{tableofcoeffs}. We
provide some examples of such linear moment combinations below.
  \item Note that by the identities~\eqref{eq:13} and~\eqref{eq:36} in the
    forthcoming proof of Theorem~\ref{thm:1} and the Cauchy-Schwarz inequality,  
\begin{equation*}
\int_E^{}U(G)\diff{\mu} \leq 
\sqrt{\int_E^{}Q^2(G)\diff{\mu}}\sqrt{\int_E^{}\left(\Gamma(G,-L^{-1}G)
    - b(G) \right)^2 \diff{\mu} },
\end{equation*}
showing that the moment combination $\int_E^{}U(G)\diff{\mu}$ indeed vanishes for a
random variable $G$ having the law $\nu$ of the target distribution (as the $\Gamma$ expression is
zero if the law of $G$ is $\nu$ by Theorem \ref{thm:5}). 
\item In order to understand the presence of the additional moment combination 
\begin{equation*}
\frac{\xi(1-b_2)}{2}\int_E^{} Q^2(G)\diff{\mu}
\end{equation*}
in the bound \eqref{eq:12}, let $\tilde{L}$ be the Markov diffusion generator of the diffusion
\eqref{eq:29} with mean $m$ and diffusion coefficient $\sigma^2(x) =
b(x)$ as in the statement of Theorem \ref{thm:1}. Then, as will be
shown in Section \ref{secchaosstruct}, the first chaos of $\tilde{L}$
always contains eigenfunctions with law $\nu$, and those
eigenfunctions have chaos grade $\tilde{\eta} = 2(1-b_2)$ by
Proposition \ref{cor:4} for $n=1$. Therefore,
$\xi = \eta - \tilde{\eta}$ measures how much the chaos grade of $G$
exceeds $\tilde{\eta}$. If $G$ is replaced by a sequence $\left\{
  G_k\colon k \geq 0 \right\}$ with
chaos grades $\left\{ \eta_k\colon k\geq 0 \right\}$, as will be done in Proposition \ref{thm:7},
then in order to converge, it is necessary that $\eta_k$ converges to $\tilde{\eta}$.
  \end{enumerate}
\end{remark}

  \begin{table}[]
    \caption{Coefficients in the linear combinations of moments}
      \label{tableofcoeffs}
\centering
\begin{tabular}{lc}
\hline
j & $c_j$                                                                                          \\ \hline
0 & $\displaystyle \frac{ \left(b_0+\frac{m (b_1+m)}{2b_2-1}\right)^2}{1-b_2}+\frac{2 m(b_1+m)^3}{3 (2 b_2-1)^3}$ \\
1 & $\displaystyle \frac{4b_0(b_1+m)}{1-2b_2}+\frac{2 (b_1+m)^2 (b_1+2m(3b_2-1))}{3 (1-2 b_2)^3}$                 \\
2 & $\displaystyle -2 b_0-\frac{2 (b_1+m)^2}{2 b_2-1}$                                                            \\
3 & $\displaystyle -2 b_1-\frac{4 m}{3}$                                                                          \\
4 & $\displaystyle \frac{1}{3}-b_2$                                                                               \\ \hline
j & $d_j$                                                                                           \\ \hline
0 & $\displaystyle \frac{\left(b_0(2b_2-1)+m (b_1+m)\right)^2}{(1-2b_2)^2(1-b_2)^2}$                                 \\
1 & $\displaystyle \frac{4 (b_1+m) (b_0 (2 b_2-1)+m(b_1+m))}{(1-2 b_2)^2 (b_2-1)}$                                \\
2 & $\displaystyle \frac{2 \left(b_0(1-2b_2)^2 + (b_1+m) (2 b_1(b_2-1)+(4 b_2-3) m)\right)}{(1-b_2)(1-2 b_2)^2}$        \\
3 & $\displaystyle \frac{4 (b_1+m)}{2 b_2-1}$                                                                     \\
4 & $\displaystyle 1$                                                                                            
\end{tabular}
\end{table}

\begin{proof}[Proof of Theorem~\ref{thm:1}]
  As $LF=-\lambda F$ and $L1=0$, we have that $LG= -\lambda (G-m)$.  Also, by definition 
  $L^{-1}G=L^{-1}F=-\frac{1}{\lambda} F = - \frac{1}{\lambda} \left( G-m
  \right)$. Therefore, also using the fact that $\Gamma$ vanishes if any of its two
  arguments is a constant, it follows that  
\begin{align*}
    \Gamma \left( G,-L^{-1}G \right)
  &=
    \frac{1}{\lambda}
    \Gamma \left( G-m,G-m \right)
  \\ &=
       \frac{1}{2\lambda} \left(L + 2\lambda \on{Id}\right)(G-m)^2
  \\ &=
       \frac{1}{2\lambda} \left( L+2\lambda \on{Id} \right)(G^2 -2mG +
       m^2).
  \end{align*}
By observing that
\begin{equation}
  \label{eq:36}
\Gamma \left( G,-L^{-1}G \right) - b(G) = \frac{1}{2\lambda}(L+2(1-b_2)\lambda\on{Id})Q(G),
\end{equation}
we can write
  \begin{align}
    \int_E^{} \left( \Gamma(G,-L^{-1}G) - b(G) \right)^2 \diff{\mu}
     &= \int_E^{} \left(
       \frac{1}{2\lambda}(L+2(1-b_2)\lambda\on{Id})Q(G) \right)^2
       \diff{\mu} \notag \\
    &= \frac{1}{4\lambda^2}\int_E^{} \left(
       (L+\eta \lambda\on{Id})Q(G) -
      \xi\lambda Q(G) \right)^2
      \diff{\mu} \notag \\
    &= \frac{1}{4\lambda^2}\left(\int_E^{} \left(
       (L+\eta\lambda\on{Id})Q(G) \right)^2
       \diff{\mu} + R_{\eta}(G)\right),\label{eq:1}
  \end{align}
where 
\begin{align}
R_{\eta}(G) &= \xi^2\lambda^2  \int_E^{} Q^2(G)\diff{\mu}
              - 2\lambda\xi \int_E^{}Q(G)(L+\eta\lambda \on{Id})Q(G)
              \diff{\mu}\notag \\
  &= 
              -2\lambda\xi \int_E^{}Q(G)(L+2(1-b_2)\lambda \on{Id})Q(G)
    \diff{\mu} -\xi^2\lambda^2 \int_E^{} Q^2(G)\diff{\mu}. \label{eq:9}
\end{align}
As $L$ is
symmetric, 
  \begin{align}
       \int_E^{} \left(
       (L+\eta\lambda\on{Id})(Q(G)) \right)^2
       \diff{\mu} &= \int_E^{} Q(G)
                    (L+\eta\lambda\on{Id})^2 Q(G) \diff{\mu} \notag \\
    &= \eta\lambda\int_E^{}
      Q(G)(L+\eta\lambda\on{Id})Q(G)\diff{\mu}\notag \\
    &+
      \int_E^{}Q(G)L(L+\eta\lambda
      \on{Id})Q(G)\diff{\mu} \notag \\
    & \leq \eta\lambda\int_E^{}
      Q(G)(L+\eta\lambda\on{Id})Q(G)\diff{\mu}\notag\\
    &= \eta\lambda\int_E^{}
      Q(G)(L+2(1-b_2)\lambda\on{Id})Q(G)\diff{\mu}\notag\\
    &+ \eta \xi \lambda^2
      \int_E^{}Q^2(G)\diff{\mu},   \label{eq:2}
  \end{align}
where the inequality follows from the fact that 
\begin{equation*}
\int_E^{}Q(G)L(L+\eta\lambda
      \on{Id})Q(G)\diff{\mu} \leq 0.
\end{equation*}
Indeed, as by assumption  
\begin{equation*}
Q(G) =
\sum_{-\kappa \in \Lambda \colon \kappa \leq \eta \lambda}\pi_{\kappa}(Q(G)),
\end{equation*}
where $\pi_{\kappa}(Q(G))$ denotes the orthogonal projection of $Q(G)$
onto the eigenspace $\on{ker}(L+\kappa \on{Id})$, one has 
\begin{align*}
&\int_E^{}Q(G)L(L+\eta\lambda
  \on{Id})Q(G)\diff{\mu} \\
 &\qquad\qquad = \sum_{-\kappa \in S \colon \kappa \leq
                               \eta\lambda}^{} \int_E^{}\pi_{\kappa}(Q(G))L(L+\eta\lambda
      \on{Id})Q(G)\pi_{\kappa}(Q(G)) \diff{\mu} \\
 &\qquad\qquad = -\sum_{-\kappa \in S \colon \kappa \leq
                               \eta\lambda}^{} \kappa(\eta\lambda -
   \kappa)\int_E^{}\pi_{\kappa}(Q(G))^2 \diff{\mu} \leq 0.
\end{align*}
Plugging \eqref{eq:2} and \eqref{eq:9} into \eqref{eq:1} yields
  \begin{align*}
    \int_E^{} \left( \Gamma(G,-L^{-1}G) - b(G) \right)^2 \diff{\mu}
     &\leq \frac{\eta-2\xi}{4\lambda}\int_E^{}
      Q(G)(L+2(1-b_2)\lambda\on{Id})Q(G)\diff{\mu} \\
    &+\frac{\xi(1-b_2)}{2}\int_E^{} Q^2(G)\diff{\mu}.
  \end{align*}
In order to prove that 
\begin{equation}
  \label{eq:13}
\int_E^{}Q(G)(L+2(1-b_2)\lambda\on{Id})Q(G)\diff{\mu} = 2 \lambda \int_E^{}U(G)\diff{\mu},
\end{equation}
we use integration by parts and the diffusion property of
$\Gamma$ to write
\begin{align*}
\int_E^{}Q(G)L Q(G)
  \diff{\mu} &= -\int_E^{}\Gamma(Q(G),Q(G))\diff{\mu} \\
&= -
  \int_E^{}(Q'(G))^2\Gamma(G,G)\diff{\mu} \\
&= -
                                            \frac{1}{6}\int_E^{}\Gamma(\left( Q'(G) \right)^3,G)\diff{\mu} \\
  &= \frac{1}{6}\int_E^{}\left( Q'(G) \right)^3LG\diff{\mu} \\
  &= -\frac{\lambda}{6}\int_E^{}\left( Q'(G) \right)^3(G-m)\diff{\mu},
\end{align*}
which concludes the proof.
\end{proof}
By combining Theorem \ref{thm:1} with Theorem \ref{thm:4}, we obtain 
quantitative moment bounds for suitable distances.

\begin{proposition}
  \label{prop:1}
In the setting and with the notation of Theorem \ref{thm:1}, let $Z$ be a random variable
with distribution $\nu$. Then, if $\eta \leq 2(1-b_2)$, one has
  \begin{equation*}
     d_{\mathscr{H}}(G,Z) \\ \leq
    c_{\mathscr{H}} \sqrt{\left( 1-b_2-\frac{\eta}{4}\right)\int_E^{}U(G)\diff{\mu}},
  \end{equation*}
whereas if $\eta > 2(1-b_2)$, one has
\begin{equation*}
     d_{\mathscr{H}}(G,Z) \\ \leq
    c_{\mathscr{H}} \sqrt{    \left( 1-b_2-\frac{\eta}{4}\right)\int_E^{}U(G)\diff{\mu}+\frac{\xi(1-b_2)}{2}\int_E^{} Q^2(G)\diff{\mu}}.
  \end{equation*}
  Here, $d_{\mathscr{H}}$ denotes an admissible distance for $G$ and $Z$,
  defined via a separating class $\mathscr{H}$ of absolutely continuous test
functions such that $\sup_{h \in \mathscr{H}}\norm{h'}_{\infty} < \infty$. The positive
constant $c_{\mathscr{H}}$ depends solely on the
class $\mathscr{H}$. 
\end{proposition}
\begin{proof}
We have to check that the function
$\sigma^2$ satisfies the assumptions of Lemma
\ref{lemmaedenviquez}. This is immediate by taking $g = \sigma^2$.
\end{proof}

At this point it is straightforward to state the following quantitative Four
Moments Theorems for approximation of any Pearson distribution admitting at
least four moments by a sequence of chaotic eigenfunctions.

\begin{theorem}
  \label{thm:7}
Let $\nu$ be a Pearson distribution associated to the
diffusion given by \eqref{eq:29} with mean $m$ and diffusion coefficient $\sigma^2(x)=b(x)=b_2x^2+b_1x+b_0$. For $k
  \in \N$, let $F_k$ be a chaotic eigenfunction with chaos grade $\eta_k$ of a Markov
  diffusion generator $L_k$ and let $G_k=F_k+m$. Furthermore, let
  $d_{\mathscr{H}}$ be an admissible distance for $\left\{ G_k \colon k \in \N
    \right\} \cup \left\{ Z \right\}$, defined via a separating class $\mathscr{H}$ of absolutely continuous test functions with uniformly
  bounded derivative. Then, if $\eta_k \leq 2(1-b_2)$, one has
  \begin{equation*}
    d_{\mathscr{H}}(G_k,Z) \leq c_{\mathscr{H}}
    \sqrt{
      \left( 1-b_2-\frac{\eta_k}{4} \right)
      \int_E^{} U(G_k) \diff{\mu}
    }
  \end{equation*}
  whereas if $\eta_k > 2(1-b_2)$, one has
  \begin{equation*}
    d_{\mathscr{H}}(G_k,Z) \leq c_{\mathscr{H}}
    \sqrt{
      \left( 1-b_2-\frac{\eta}{4} \right)
      \int_E^{} U(G_k) \diff{\mu}
      +
      \frac{\xi_k (1-b_2)}{2}
      \int_E^{} Q^2(G_k) \diff{\mu}
    },
  \end{equation*}
  where $\xi_k = \eta_k - 2(1-b_2)$. Here, $c_{\mathscr{H}}$ is a positive constant solely depending on the
  separating class $\mathscr{H}$. In particular, the following two conditions are
  sufficient for the sequence $\left\{ G_k\colon n\geq 0 \right\}$ to converge in distribution to $Z$:
  \begin{enumerate}[(i)]
  \item 
    \begin{equation}
\label{eq:44}
      \int_E^{} U(G_k) \diff{\mu} \to 0.
    \end{equation}
  \item For every subsequence $(\eta_{k_{r}})$ of $(\eta_k)$ such that
    $\eta_{k_r}>2(1-b_2)$ for every $r \in \N$ one has
    \begin{equation*}
      \sup_{r \in \N}
      \int_E^{} Q^2(G_{k_r}) \diff{\mu} < \infty
    \end{equation*}
    and $\eta_{k_r} \to 2(1-b_2)$.
\end{enumerate}
\end{theorem}

\begin{example}
  Let us give some explicit examples of the moment combinations appearing in
  Condition~\eqref{eq:44} for several target distributions. To improve readability, we abbreviate the $p$-th
  moment $\int_E^{} G_k^p \diff{\mu}$ by $m_p(G_k)$.
  \begin{enumerate}[(i)]
\item For convergence towards a centered Gaussian distribution with
  variance $\sigma^2$, we have $b(x)=\sigma^2$, so that by Table
  \ref{tableofcoeffs}, we get that $c_0 = \sigma^4$, $c_1 = 0$, $c_2 =
  -2\sigma^2$, $c_3=0$ and $c_4 =\frac{1}{3}$, hence recovering the well-known moment condition
  \begin{equation*}
    \frac{1}{3} m_4(G_k) - 2\sigma^2 m_{2}(G_k) +  \sigma^4 \to 0,
  \end{equation*}
  which becomes $m_4(G_k) \to 3$ when $m_2(G_k)=\sigma^2 = 1$.
  \item For a (heavy-tailed) Student $t$-distribution with mean zero and $\tau$ degrees of
freedom (which is a particular case of a Skew $t$-distribution with
parameters $m=\frac{\tau+1}{2}$, $\lambda = \nu = 0$ and $\alpha =
\sqrt{\tau}$), we have $b(x) = \frac{x^2}{\tau-1} +
\frac{\tau}{\tau-1}$. Therefore, the moment condition becomes
  \begin{equation*}
    \frac{(\tau-4)}{3(\tau-1)} m_4(G_k)
    - \frac{2\tau}{(\tau-1)} m_2(G_k) +  \frac{\tau^2}{\tau^2-3\tau+2} \to 0.
  \end{equation*}
  This moment condition is new.
  \item For the inverse gamma distribution
    with shape parameter $\alpha>0$ and scale parameter $\beta>0$, which is non-centered (as
    opposed to the two previous examples) with mean
    $\frac{\beta}{\alpha -1}$, we have $b(x)=
    \frac{x^2}{\alpha-1}$. We hence obtain new moment conditions as
    well, ensuring convergence to the (heavy-tailed) inverse gamma distribution. For instance, setting
    the shape parameter $\alpha = 5$, we get that
  \begin{equation*}
    \frac{1}{12} m_4(G_k)
    - \frac{\beta}{3} m_3(G_k) + \frac{\beta^2}{4} m_2(G_k)- \frac{\beta^3}{24} m_1(G_k) \to 0.
  \end{equation*}
\end{enumerate}
\end{example}

\section{Pearson chaos}
\label{secchaosstruct}

As an application of our results, we treat the case where the
converging sequence of chaotic eigenfunctions itself comes from a
generator associated to a Pearson law. To avoid technicalities, we
present here only the finite-dimensional case, analogous results in
infinite dimension can be obtained in a similar way. We begin by
describing a general and well-known tensorization procedure of Markov generators.

Fix $N \geq 2$ and, for $1 \leq i \leq N$, let $\mathcal{L}_i$ be a generator with
invariant probability measure $\mu_i$ and $L^2$-domain $\mathcal{D}(\mathcal{L}_i)
\subseteq L^{2}(E_i,\mathcal{F}_i,\mu_i)$.
Let $(E,\mathcal{F},\mu)$ be the product of the probability spaces
$(E_i,\mathcal{F}_i,\mu_i)$. Then we can define a generator $L_{N}= \otimes_{i=1}^N
\mathcal{L}_i$ on $\mathcal{D}(L_N)= \bigotimes_{i=1}^N
\mathcal{D}(\mathcal{L}_i)$ by
\begin{equation*}
  L_N 
  \left( F_1 \times F_2 \times \dots \times F_N \right)
  =
  \sum_{i=1}^N F_1 \times \dots \times F_{i-1} \times \left( \mathcal{L}_iF_i \right) \times F_{i+1} \times \dots \times F_N.
\end{equation*}
From this definition, it follows that if $F_i$ is an eigenfunction of
$\mathcal{L}_i$ with eigenvalue $\lambda_i$, then $F = \oplus_{i=1}^N F_i$ is an
eigenfunction of $L_N$ with eigenvalue $\lambda = \sum_{i=1}^N \lambda_i$. The following corollary describes how the chaos grade behaves under tensorization.

\begin{corollary}
  \label{cor:1}
  In the above setting, let each eigenfunction $F_i$ be chaotic with
  chaos grade $\eta_i$. Then $F$ is chaotic and its chaos grade $\eta$
  is bounded as follows:
  \begin{equation*}
    \min \left\{ \eta_1,\eta_2,\dots,\eta_N \right\}
    \leq
    \eta
    \leq
    \max \left\{ \eta_1, \eta_2,\dots,\eta_N \right\}.
  \end{equation*}
The above inequalities become equalities, if, and only if, all of the chaos
grades $\eta_i$ are equal. 
\end{corollary}

\begin{proof}
  By definition, the squares $F_i^2$ can be expanded as sums of eigenfunctions,
  with the eigenvalue of largest magnitude in such an expansion being
   $\lambda_{i} \eta_i$. Therefore, $F^2$ can also be expanded as a sum of eigenfunctions, with the eigenvalue of largest magnitude,  say  $\lambda_{\text{max}}$, being given by
  \begin{equation*}
    \lambda_{\text{max}} = \sum_{i=1}^N \lambda_{i} \eta_i.
  \end{equation*}
  Applying the definition of chaos grade (see Definition \ref{def:1}) now yields that
  \begin{equation*}
    \eta = \frac{\lambda_{max}}{\lambda} = \frac{\sum_{i=1}^{N} \lambda_{i} \eta_i}{\sum_{i=1}^{N} \lambda_{i}},
  \end{equation*}
  from which the assertion follows as all $\lambda_i$ have the same sign.
\end{proof}

In the following proposition, we calculate the possible range of
values of the chaos grade for eigenfunctions related to all six Pearson distributions. 
\begin{proposition}
  \label{cor:4}
  Let $L$ be the generator associated to a Pearson diffusion defined
  by \eqref{eq:29} and denote the eigenvalues of $L$ by $-\lambda_n$
  where $\lambda_n=n(1-(n-1)b_2)\theta$ for $b_2 < \frac{1}{2n-1}$. Let $F_n$ be an eigenfunction of $L$ with respect to $-\lambda_n$. Then $F_n$ is chaotic, if, and only if, $b_2 < \frac{1}{4n-1}$, and in this case its chaos grade $\eta_n$ is given by
  \begin{equation}
    \label{eq:20}
  \eta_n = \eta_n(b_2) =
  \begin{cases}
    2 &\qquad \text{if $b_2=0$,} \\
    2 \left(
    1 + \frac{n}{n-1-\frac{1}{b_2}}
  \right)
  & \qquad \text{if $b_2 \neq 0$.} \\
\end{cases}
\end{equation}
Furthermore, the following is true.
\begin{enumerate}[(i)]
\item If $\mu$ is a Student, $F$- or inverse Gamma distribution, then $\eta_n \in \left( \frac{4}{3},2-2b_2 \right]$.
\item If $\mu$ is a Gaussian or Gamma distribution, then $\eta_n=2$. 
\item If $\mu$ is a Beta distribution then $\eta_n \in (4  ,2-2b_2]$, if $b_2<-1$, $\eta_n=4$, if $b_2=-1$ and $\eta_n \in [2-2b_2,4)$, if $-1<b_2<0$.
\end{enumerate}
\end{proposition}

\begin{proof}
  An eigenfunction $F_n$ of a Pearson generator with respect to the
  eigenvalue $-\lambda_{n} = n (1- (n-1)b_2)\theta$ is an orthogonal polynomial of degree
  $n$. Its square is then a polynomial of degree $2n$. In order for
  $F_n^2$ to be expressable as a sum of square integrable
  eigenfunctions, we therefore need that the first $2n$ eigenfunctions
  of $L$ are square integrable, or equivalently that moments up to
  order $4n$ exist. Hence, by \eqref{eq:19}, the condition required is
  \begin{equation}
    \label{eq:4}
    b_2 < \frac{1}{4n-1}.
  \end{equation}
  Let us assume that the above inequality is
  satisfied. Then, by its very definition, $\eta_n$ is given by the
  quotient of the $2n$-th eigenvalue with the $n$-th one. Indeed, as $\eta_n$
  is the multiplicative factor that indicates what eigenvalue the
  highest-order eigenfunction in the decomposition of the square of $F_n$ is
  associated to. On the other hand, we know that the square of the
  polynomial eigenfunction of degree $n$ produces a sum of polynomial
  eigenfunctions up to degree $2n$, corresponding to the eigenvalue
  $-\lambda_{2n}$. Hence we have 
  \begin{equation}
    \label{eq:3}
  \eta_n = \frac{\lambda_{2n}}{\lambda_n}=\frac{2n(1-(2n-1)b_2)\theta}{n(1-(n-1)b_2)\theta},
\end{equation}
so that~\eqref{eq:20} follows. Assertion (ii) is immediate as in this
case $b_2=0$ and the chaos grade is constant. In order
to show assertion (i) in which $b_2>0$, 
note that the function $n \mapsto \eta_n(b_2)$ is
decreasing. Therefore, the largest possible chaos grade is obtained by
taking $n=1$ in \eqref{eq:3}, which gives $2(1-b_2)$. On the other
hand, as by \eqref{eq:4}, $n < \frac{1}{4}\left( \frac{1}{b_2}+1
\right)$, the lower bound $\frac{4}{3}$ of the chaos grade is obtained by taking 
$n=\left\lfloor \frac{1}{4}\left( \frac{1}{b_2}+1 \right)
\right\rfloor$. Assertion (iii) where $b_2<0$ follows in a similar way.

\end{proof}
Proposition~\ref{cor:4} shows that on a global level, the chaos grade $\eta$ of
chaotic eigenfunctions of a Pearson generator lies in the interval $\left(
  \frac{4}{3},\infty  \right)$. Furthermore, all values in this interval can be
attained, in the sense that if $x$ is such a value, then there exists a
generator $L$ of a Pearson diffusion \eqref{eq:29} which has a chaotic eigenfunction of chaos grade $x$. The six types of Pearson distributions are partitioned into three classes with disjoint intervals for the chaos grade values of the corresponding eigenfunctions. These intervals are all of the form
\begin{equation*}
  \left\{ 2(1-b_2) \colon b_2 \in I \right\},
\end{equation*}
where $I$ is the set of allowed values for the corresponding class, i.e., $I=(-\infty,0)$ for the class of student, $F$- and inverse Gamma distributions, $I=\left\{ 0 \right\}$ for Gaussian and Gamma distributions and $I=(0,\infty)$ for the Beta distributions.

\begin{figure}[h]
  \centering
\begin{tikzpicture}[every edge/.style={shorten <=1pt, shorten >=1pt}]
  \draw[->] (3.75,0) -- (4,0)  node [below, above=-24pt] {$\frac{4}{3}$} --
(6,0) node
  [below, above=-20pt] {2} -- (12,0) node [below, above=-20pt] {4} --
(12.75,0);
  \draw [<->] (4,0.2) -- (5.9,0.2) node [midway, above=2pt,
  anchor=south] {$\substack{\rm
      Student\\ \rm F-dist.\\ \rm Inv.\ Gamma}$} ;
  \draw [<-] (6.1,0.2) -- (12.2,0.2) node [midway, above=2pt,
  anchor=south] {$\substack{\rm Beta}$} ;
  \draw (4,0.07) -- (4,-0.07);
  \draw (6,0.07) -- (6,-0.07);
  \draw (12,0.07) -- (12,-0.07);
  \draw [->] (6,1.3) node [above] {$\substack{\rm
      Gaussian\\ \rm Gamma}$} -- (6,0.2);
  \draw [dotted] (12.2,0.2) -- (12.75,0.2);
\end{tikzpicture}
\caption{Possible chaos grades for the Pearson distributions}
\label{chaosgradepicture}
\end{figure}
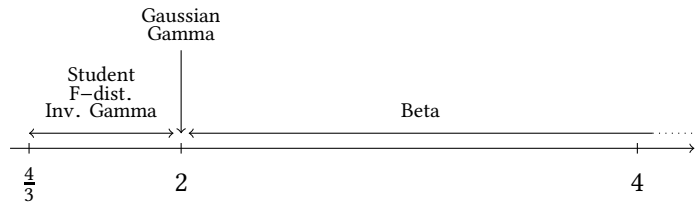

Applying the tensorization procedure described above to the case 
where all generators $\mathcal{L}_i$ are equal to some generator $L$ of a
Pearson diffusion immediately yields the following result.

\begin{theorem}
  Let $\mu$ be a Pearson distribution and
  $L$ be  the associated Markov generator. Denote its
  eigenvalues by $\left\{ -\lambda_i
    \colon 0 \leq i < I  \right\}$, where $I \in \N \cup \left\{ \infty \right\}$ and such
  that $\lambda_i < \lambda_{i+1}$. Furthermore, denote by $P_i$ the $i$-th orthogonal polynomial associated to $\mu$.
  Let $L_N = L^{\otimes N}$ be the generator obtained by the tensorization procedure described above and denote by $\mu_N$ the
  associated product measure. Then the set of
  eigenvalues of $L_N$ is given by
    \begin{equation*}
      S = \left\{ - \sum_{i=1}^N \lambda_{k_i} \colon k_1,\dots,k_N \in I \right\}.
    \end{equation*}
    If $-\lambda= - \sum_{i=1}^N \lambda_{k_i}$ is such an eigenvalue, then all eigenfunctions $F$
    of $L_N$ with respect to $-\lambda$ are of the form
    \begin{equation*}
      F = \sum_{\abs{\alpha}=p}^{} a_{\alpha} P_{\alpha},
    \end{equation*}
    where
    \begin{enumerate}[(i)]
    \item $p=\sum_{i=1}^N k_i$,
    \item the sum is taken over all $N$-dimensional multi-indices $\alpha=(\alpha_1,\dots,\alpha_N)$ of order $p$,
    \item the $a_{\alpha}$ are real constants,
    \item $P_{\alpha}(x)=P_{\alpha}(x_1,x_2,\dots,x_N) = \prod_{i=1}^{N} P_{\alpha_i}(x_i)$.
    \end{enumerate}
\end{theorem}

Combining Corollary~\ref{cor:1} with Proposition~\ref{cor:4} and the discussion
thereafter, we see that for the six classes of Pearson distributions the
intervals for the chaos grades of the respective chaotic eigenfunctions are invariant under tensorization. In other words, the chaos grades of chaotic eigenfunctions of $L_N$
\begin{enumerate}[(i)]
\item assume values in the interval $\left( \frac{4}{3},2 \right)$, if the
  tensorized distribution is Student, $F$- or inverse Gamma,
\item are equal to two in the case of tensorized Gaussian or Gamma
  distributions,
\item lie in the interval $\left( 2,\infty \right)$ if the distribution is Beta.
\end{enumerate}
Coming back to the Four Moments Theorems proved in Section~\ref{s-1}, the
possible chaos grades also yield a heuristic about ``compatible'' Pearson
distributions, in the sense that one can be obtained as a limit of a
chaos of another Pearson distribution. Recall from Section~\ref{s-1} (in particular Remark~\ref{s-6}-(iii)) that if we
want to approximate a random variable $Z$ with a Pearson law and chaos
grade $\tilde{\eta}$ to be the limit of a sequence $(G_n)$ of chaotic random variables with corresponding chaos grade sequence
$(\eta_n)$, we need that $\eta_n \leq \widetilde{\eta}$ or $\eta_n \to \widetilde{\eta}$, where
$\widetilde{\eta}$ is the chaos grade of $Z$ when seen as a chaotic random variable
itself. For example, if $Z$ has a Gaussian or Gamma distribution, then
$\widetilde{\eta}=2$. Therefore, chaotic random variables coming from
a heavy tailed Pearson chaos are compatible, as in this case we always
have $\eta_n \leq
2$. The Gamma and Gaussian chaos is of course compatible as well as here the two
chaos grades coincide and for convergence
from Beta chaos to a Gaussian or Gamma distribution, our conditions require that
$\eta_n \to 2$. This translates to the parameters of the underlying invariant Beta
measure growing to infinity.
Taking $Z$ to be a heavy tailed Pearson distribution yields a chaos grade
$\widetilde{\eta}$ which is strictly less than two. Here, our
conditions suggest that only heavy-tailed chaos are compatible.
The aforementioned heuristic could likely be made rigorous by a detailed study
of the carré du champ characterization given in Theorem~\ref{thm:5} and is
left for future research.

\bibliography{/home/solesne/Dropbox/work/research/biblio} 

\providecommand{\bysame}{\leavevmode\hbox to3em{\hrulefill}\thinspace}
\providecommand{\MR}{\relax\ifhmode\unskip\space\fi MR }
\providecommand{\MRhref}[2]{%
  \href{http://www.ams.org/mathscinet-getitem?mr=#1}{#2}
}
\providecommand{\href}[2]{#2}
\begin{thebibliography}{BSSr05}

\bibitem[ACP14]{azmoodeh_fourth_2014}
Ehsan Azmoodeh, Simon Campese, and Guillaume Poly, \emph{Fourth {Moment}
  {Theorems} for {Markov} diffusion generators}, Journal of Functional Analysis
  \textbf{266} (2014), no.~4, 2341--2359. \MR{3150163}

\bibitem[AL{\v S}13]{avram_spectral_2013}
F.~Avram, N.~N. Leonenko, and N.~{\v S}uvak, \emph{On spectral analysis of
  heavy-tailed {Kolmogorov}-{Pearson} diffusions}, Markov Processes and Related
  Fields \textbf{19} (2013), no.~2, 249--298. \MR{3113945}

\bibitem[Bak14]{bakry_symmetric_2014}
Dominique Bakry, \emph{Symmetric diffusions with polynomial eigenvectors},
  Stochastic analysis and applications 2014, Springer {Proc}. {Math}. {Stat}.,
  vol. 100, Springer, Cham, 2014, DOI: 10.1007/978-3-319-11292-3\_2,
  pp.~25--49. \MR{3332708}

\bibitem[B{\'E}85]{bakry_diffusions_1985}
D.~Bakry and Michel {\'E}mery, \emph{Diffusions hypercontractives},
  S{\'e}minaire de probabilit{\'e}s, {XIX}, 1983/84, Lecture {Notes} in
  {Math}., vol. 1123, Springer, Berlin, 1985, pp.~177--206. \MR{889476}

\bibitem[BGL14]{bakry_analysis_2014}
Dominique Bakry, Ivan Gentil, and Michel Ledoux, \emph{Analysis and geometry of
  {Markov} diffusion operators}, Grundlehren der {Mathematischen}
  {Wissenschaften} [{Fundamental} {Principles} of {Mathematical} {Sciences}],
  vol. 348, Springer, Cham, 2014, DOI: 10.1007/978-3-319-00227-9. \MR{3155209}

\bibitem[BH91]{bouleau_dirichlet_1991}
Nicolas Bouleau and Francis Hirsch, \emph{Dirichlet forms and analysis on
  {Wiener} space}, de {Gruyter} {Studies} in {Mathematics}, vol.~14, Walter de
  Gruyter \& Co., Berlin, 1991. \MR{1133391}

\bibitem[BSSr05]{bibby_diffusion-type_2005}
Bo~Martin Bibby, Ib~Michael Skovgaard, and Michael S{\o}~rensen,
  \emph{Diffusion-type models with given marginal distribution and
  autocorrelation function}, Bernoulli. Official Journal of the Bernoulli
  Society for Mathematical Statistics and Probability \textbf{11} (2005),
  no.~2, 191--220. \MR{2132002}

\bibitem[Dud02]{dudley_real_2002}
R.~M. Dudley, \emph{Real analysis and probability}, Cambridge {Studies} in
  {Advanced} {Mathematics}, vol.~74, Cambridge University Press, Cambridge,
  2002. \MR{1932358}

\bibitem[EVq15]{eden_nourdin-peccati_2015}
Richard Eden and Juan V{\textbackslash}'{\i}~quez, \emph{Nourdin-{Peccati}
  analysis on {Wiener} and {Wiener}-{Poisson} space for general distributions},
  Stochastic Processes and their Applications \textbf{125} (2015), no.~1,
  182--216. \MR{3274696}

\bibitem[FOT11]{fukushima_dirichlet_2011}
Masatoshi Fukushima, Yoichi Oshima, and Masayoshi Takeda, \emph{Dirichlet forms
  and symmetric {Markov} processes}, extended ed., De {Gruyter} {Studies} in
  {Mathematics}, vol.~19, Walter de Gruyter \& Co., Berlin, 2011. \MR{2778606}

\bibitem[FS08]{forman_pearson_2008}
Julie~Lyng Forman and Michael S{\o}rensen, \emph{The {Pearson} {Diffusions}:
  {A} {Class} of {Statistically} {Tractable} {Diffusion} {Processes}},
  Scandinavian Journal of Statistics \textbf{35} (2008), no.~3, 438--465 (en).

\bibitem[JKB94]{johnson_continuous_1994}
Norman~L. Johnson, Samuel Kotz, and N.~Balakrishnan, \emph{Continuous
  univariate distributions. {Vol}. 1}, second ed., Wiley {Series} in
  {Probability} and {Mathematical} {Statistics}: {Applied} {Probability} and
  {Statistics}, John Wiley \& Sons, Inc., New York, 1994. \MR{1299979}

\bibitem[JKB95]{johnson_continuous_1995}
\bysame, \emph{Continuous univariate distributions. {Vol}. 2}, second ed.,
  Wiley {Series} in {Probability} and {Mathematical} {Statistics}: {Applied}
  {Probability} and {Statistics}, John Wiley \& Sons, Inc., New York, 1995.
  \MR{1326603}

\bibitem[KT12]{kusuoka_steins_2012}
Seiichiro Kusuoka and Ciprian~A. Tudor, \emph{Stein's method for invariant
  measures of diffusions via {Malliavin} calculus}, Stochastic Processes and
  their Applications \textbf{122} (2012), no.~4, 1627--1651. \MR{2914766}

\bibitem[Led12]{ledoux_chaos_2012}
Michel Ledoux, \emph{Chaos of a {Markov} operator and the fourth moment
  condition}, The Annals of Probability \textbf{40} (2012), no.~6, 2439--2459
  (EN), Zentralblatt MATH identifier: 06114704.

\bibitem[Maz97]{mazet_classification_1997}
Olivier Mazet, \emph{Classification des semi-groupes de diffusion sur
  \${\textbackslash}bf {R}\$ associ{\'e}s {\`a} une famille de polyn{\^o}mes
  orthogonaux}, S{\'e}minaire de {Probabilit{\'e}s}, {XXXI}, Lecture {Notes} in
  {Math}., vol. 1655, Springer, Berlin, 1997, pp.~40--53. \MR{1478714}

\bibitem[NOL08]{nualart_central_2008}
David Nualart and Salvatore Ortiz-Latorre, \emph{Central limit theorems for
  multiple stochastic integrals and {Malliavin} calculus}, Stochastic Processes
  and their Applications \textbf{118} (2008), no.~4, 614--628. \MR{2394845}

\bibitem[NP05]{nualart_central_2005}
David Nualart and Giovanni Peccati, \emph{Central limit theorems for sequences
  of multiple stochastic integrals}, The Annals of Probability \textbf{33}
  (2005), no.~1, 177--193. \MR{2118863}

\bibitem[NP09a]{nourdin_noncentral_2009}
Ivan Nourdin and Giovanni Peccati, \emph{Noncentral convergence of multiple
  integrals}, The Annals of Probability \textbf{37} (2009), no.~4, 1412--1426
  (EN). \MR{MR2546749}

\bibitem[NP09b]{nourdin_steins_2009}
\bysame, \emph{Stein's method on {Wiener} chaos}, Probability Theory and
  Related Fields \textbf{145} (2009), no.~1-2, 75--118. \MR{2520122}

\bibitem[NP12]{nourdin_normal_2012}
\bysame, \emph{Normal approximations with {Malliavin} calculus}, Cambridge
  {Tracts} in {Mathematics}, vol. 192, Cambridge University Press, Cambridge,
  2012, From Stein's method to universality.

\bibitem[Nua06]{nualart_malliavin_2006}
David Nualart, \emph{The {Malliavin} calculus and related topics}, second ed.,
  Probability and its {Applications} ({New} {York}), Springer-Verlag, Berlin,
  2006.

\bibitem[Pea95]{pearson_contributions_1895}
Karl Pearson, \emph{Contributions to the mathematical theory of evolution -
  {II}. {Skew} variation in homogeneous material}, Phil. Trans. R. Soc. Lond. A
  \textbf{186} (1895), 343--414 (en).

\bibitem[Ste86]{stein_approximate_1986}
Charles Stein, \emph{Approximate computation of expectations}, Institute of
  {Mathematical} {Statistics} {Lecture} {Notes}{\textemdash}{Monograph}
  {Series}, 7, Institute of Mathematical Statistics, Hayward, CA, 1986.
  \MR{882007}

\end{thebibliography}
\end{document}